\documentclass{article}

\usepackage{amscd,amssymb,verbatim,graphicx,stmaryrd,amsthm,color,amsmath} 
\usepackage[colorlinks]{hyperref}
\usepackage{diagrams}
\usepackage[margin=1.7in]{geometry}

\input xy    %Loads xy-Pic
\xyoption{curve}  %Inputs curve things (See section 8 of manual)
\xyoption{arc} %Fancy version of the previous
\xyoption{frame} %Frame extension. Needed for frm command
\xyoption{tips} %For fancy arrows
\xyoption{arrow} %Needed for \ar command 
\xyoption{color} %Color

\usepackage{mathpazo} %Changes font to something nice.

\newtheorem{theorem}{Theorem}[section]  %This keeps track of the section in the naming of the theorems
\newtheorem{lemma}[theorem]{Lemma}  %This creates a lemma environment and numbers it along with the theorems
\newtheorem{proposition}[theorem]{Proposition}
\newtheorem{corollary}[theorem]{Corollary}

\newtheorem{remark}[theorem]{Remark}

\usepackage{blindtext}
\usepackage{changepage}
\newtheoremstyle{indented}{3pt}{3pt}{}{}{\bfseries}{.}{.5em}{}
\theoremstyle{indented}
\newtheorem{innercustomthm}{}
\newenvironment{customthmQuant}[1]{\renewcommand\theinnercustomthm{#1}\innercustomthm}{\endinnercustomthm}

\newcommand{\End}{\mathrm{End}}

\newcommand{\limd}[1]{\displaystyle \lim_{#1}}
\newcommand{\intd}[1]{\displaystyle \int_{#1}}
\newcommand{\intdd}[2]{\displaystyle \int_{#1}^{#2}}
\newcommand{\fracd}[2]{\displaystyle \frac{#1}{#2}}
\newcommand{\bb}[1]{\mathbb{#1}}

\newcommand{\defeq}{\mathrel{\mathpalette{\vcenter{\hbox{$:$}}}=}}
\newcommand{\SU}{\mathrm{SU}}
\newcommand{\SO}{\mathrm{SO}}
\newcommand{\U}{\mathrm{U}}
\newcommand{\GL}{\mathrm{GL}}
\newcommand{\SL}{\mathrm{SL}}
\newcommand{\PU}{\mathrm{PU}}

\newcommand{\fl}{{\small \mathrm{flat}}}

\newcommand{\A}{{\mathcal{A}}}
\newcommand{\G}{{\mathcal{G}}}

\newcommand{\rheat}{\mathrm{Heat}}

\newcommand{\CS}{{\mathcal{CS}}}

\newcommand{\eps}{\epsilon}
\newcommand{\afA}{A}

\newcommand{\afa}{a}

\newcommand{\afu}{u}
\title{The Chern-Simons invariants for the double \\ of a compression body}

\author{David L. Duncan}

\date{}

\begin{document}

\maketitle

\begin{abstract}
Given a 3-manifold that can be written as the double of a compression body, we compute the Chern-Simons critical values for arbitrary compact connected structure groups. We also show that the moduli space of flat connections is connected when there are no reducibles.
\end{abstract}

\tableofcontents

\section{Introduction}

Let $G$ be a Lie group with Lie algebra $\frak{g}$. Given a principal $G$-bundle $P \rightarrow Y$ over a closed, oriented 3-manifold $Y$, one can define the Chern-Simons function 
$$\CS: \A(P) \rightarrow \bb{R} / \bb{Z},$$ 
where $\A(P)$ is the space of connections on $P$. The set of critical points of $\CS$ is the space of flat connections $\A_\fl(P) \subset \A(P)$, and the critical values are topological invariants of $Y$. In general, computing the critical values of $\CS$ is fairly difficult. Nevertheless, various techniques have been developed to handle certain classes of 3-manifolds; for example, see \cite{KK}, \cite{Auckly}, \cite{Rez}, \cite{Nishi}, \cite{NY}, \cite{DS3}, \cite{Wenergyidentity}. Most of these techniques require are specific to the choice of Lie group $G$, common examples being $\SU(2)$ and $\SL_{\bb{C}}(2)$.

In the present paper we compute the Chern-Simons critical values for any 3-manifold $Y$ that can be written as a \emph{double}
$$Y = \overline{H} \cup_{\partial H} H,$$ 
where $H$ is a compression body, $\overline{H}$ is a copy of $H$ with the opposite orientation, and the identity map on $\partial H$ is used to glue $\overline{H}$ and $H$; see Figure \ref{figure1}. For us, the term \emph{compression body} means that 

\begin{itemize}
\item $H$ is a compact, connected, oriented cobordism between surfaces $\Sigma_-, \Sigma_+$,
\item $H$ admits a Morse function $f: H \rightarrow \left[-1, 1 \right]$ with critical points of index 0 or 1, 
\item all critical values of $f$ are in the interior of $(-1, 1)$, and 
\item $f^{-1}(\pm 1) = \Sigma_\pm$. 
\end{itemize}
It follows that, up to homotopy, $H$ can be obtained from $\Sigma_+$ by attaching 2-handles. These topological assumptions imply that $\Sigma_+$ is connected; there is no bound on the number of components of $\Sigma_-$. (Note that not every 3-manifold can be realized as the double of a compression body; the Poincar\'{e} homology sphere is a simple counterexample.) 

Throughout this paper we work with an arbitrary compact, connected Lie group $G$, and we assume the bundle $P$ is obtained by doubling a bundle over $H$ in the obvious way.

\begin{figure}[h]
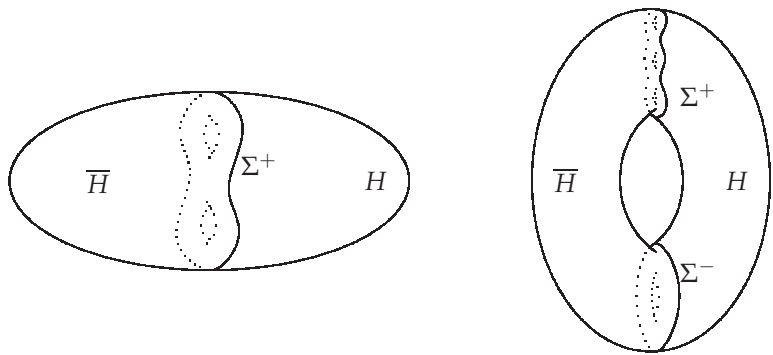
\label{figure1}

\vspace{2.1cm}

		\centerline{\xy 
		(-30,0)*={\xy
		/r.175pc/:,
	 (0,0)*\ellipse(36,16.06){-};
(0,16)*{}="A";
(0,-16)*{}="B";
"A";"B" **\crv{~**\dir{.} (-2,16)& (-8,9)& (-1,2)& (-9, -11)& (-2,-16)};  %Left part of \Sigma
"A";"B" **\crv{ (2,16.5)& (9,11)& (1,-2)& (8, -9)& (2,-16.5)};  %Right part of \Sigma
(0,-4.5)*{}="b1";
(0,-10.5)*{}="b2";
"b1";"b2" **\crv{~**\dir{.} (-3, -8.5)}; %left bottom hole in \Sigma
(-.5,-3.75)*{}="B1";
(-.5,-11.25)*{}="B2";
"B1";"B2" **\crv{~**\dir{.} (3, -8.5)}; %right part of bottom hole in \Sigma
(0.5,5)*{}="t1";
(0.5,11)*{}="t2";
"t1";"t2" **\crv{~**\dir{.} (-2.5, 9)}; %left part of top hole
(0,4.25)*{}="T1";
(0,11.75)*{}="T2";
"T1";"T2" **\crv{~**\dir{.} (3.5,9)}; %right part of top hole
(9,3)*{ \Sigma^+};
(-20,0)*{\overline{H}};
(30,0)*{H};
\endxy};
(25,0)*={\xy
		%%%%%%%%%%%%Left 3-manifold
		/r.15pc/:,
	 (0,0)*\ellipse(25,36){-};
	 (1,15)*{}="t1";
(1,-15)*{}="t2";
"t1";"t2" **\crv{(-9,8)& (-9,-8)}; %bottom part of hole
 (-.2,14)*{}="t1";
(-.2,-14)*{}="t2";
"t1";"t2" **\crv{(9,8)& (9,-8)}; %top part of hole
%%%%%%%%%\Sigma_1
(1,-13.3)*{}="A1";
(1,-36)*{}="B1";
"A1";"B1" **\crv{(3,-13.3)& (9,-24) (3,-36)};  %Top part of \Sigma_1
"A1";"B1" **\crv{~**\dir{.}  (-1,-13.3)& (-5,-24)& (-1,-36)};  %Bottom part of \Sigma_1
(.7,-20.5)*{}="tl11";
(.7,-28.5)*{}="tl12";
"tl11";"tl12" **\crv{~**\dir{.} (2,-23)& (2,-26)}; %Hole top in \Sigma_1
(1,-19.5)*{}="tr11";
(1,-29.5)*{}="tr12";
"tr11";"tr12" **\crv{~**\dir{.} (0,-22)& (0,-26)}; %Hole bottom in \Sigma_1
%%%%%%%%%\Sigma_2
%(24.75,-4)*{}="A2";
%(6.4,-4)*{}="B2";
%"A2";"B2" **\crv{~**\dir{.} (24.75,-6)& (20,-8)& (15,-5)& (10,-8)& (6,-5)};  %Left part of \Sigma_2
%"A2";"B2" **\crv{ (25,-3)& (20,-1)& (15,-4)& (10,-1)& (6,-3)};  %Right part of \Sigma_2
% (22,-5)*{}="tl21";
%(19,-5)*{}="tl22";
%"tl21";"tl22" **\crv{~**\dir{.} (21,-4)& (20,-4)}; %Top hole right side in \Sigma_2
% (21.5,-4.7)*{}="tr21";
%(19.5,-4.7)*{}="tr22";
%"tr21";"tr22" **\crv{~**\dir{.} (21,-5.2)& (20,-5.2)}; %Top hole left side in \Sigma_2
% (12,-5)*{}="bl21";
%(9,-5)*{}="bl22";
%"bl21";"bl22" **\crv{~**\dir{.} (11,-4)& (10,-4)}; %Bottom hole right side in \Sigma_2
% (11.5,-4.7)*{}="br21";
%(9.5,-4.7)*{}="br22";
%"br21";"br22" **\crv{~**\dir{.} (11,-5.2)& (10,-5.2)}; %Bottom hole left side in \Sigma_2
%%%%%%%%%\Sigma_3
(1,13.3)*{}="A3";
(1,36)*{}="B3";
"A3";"B3" **\crv{ (3,13.3)& (4,17.41)& (1,21.52)& (4,25.63)& (1,29.74)& (4,32.9)& (3,36)};  %Top part of \Sigma_3
"A3";"B3" **\crv{~**\dir{.}  (0,13.3)& (-2,17.41)& (0,21.52)& (-2,25.63)& (0,29.74)& (-2,32.9)& (0,36)};  %Bottom part of \Sigma_3
(.8,33.9)*{}="lt31";
(.8,32.3)*{}="lt32";
"lt31";"lt32" **\crv{~**\dir{.} (1.5,33.1)}; %Left hole top side in \Sigma_3
(1,34.3)*{}="lb31";
(1,31.8)*{}="lb32";
"lb31";"lb32" **\crv{~**\dir{.} (0,33.05)}; %Left hole bottom side in \Sigma_3
(.8,26)*{}="mt31";
(.8,24.08)*{}="mt32";
"mt31";"mt32" **\crv{~**\dir{.} (1.5,24.88)}; %Middle hole top side in \Sigma_3
(1,26.8)*{}="mb31";
(1,23.6)*{}="mb32";
"mb31";"mb32" **\crv{~**\dir{.} (0,24.83)}; %Middle hole bottom side in \Sigma_3
(.8,18.28)*{}="rt31";
(.8,16.46)*{}="rt32";
"rt31";"rt32" **\crv{~**\dir{.} (1.5,17.26)}; %Right hole top side in \Sigma_3
(1,19.18)*{}="rb31";
(1,15.98)*{}="rb32";
"rb31";"rb32" **\crv{~**\dir{.} (0,17.21)}; %RIght hole bottom side in \Sigma_3
%%%%%%%%%\Sigma_4
%(-6.25,4)*{}="A4";
%(-24.75,4)*{}="B4";
%"A4";"B4" **\crv{~**\dir{.} (-6.25,2)& (-11,0)& (-16,3)& (-21,0)& (-25,3)};  %Left part of \Sigma_4
%"A4";"B4" **\crv{ (-6,5)& (-11,7)& (-16,4)& (-21,7)& (-25,5)};  %Right part of \Sigma_4
% (-9,3)*{}="tl41";
%(-12,3)*{}="tl42";
%"tl41";"tl42" **\crv{~**\dir{.} (-10,4)& (-11,4)}; %Top hole right side in \Sigma_4
% (-9.5,3.3)*{}="tr41";
%(-11.5,3.3)*{}="tr42";
%"tr41";"tr42" **\crv{~**\dir{.} (-10,2.8)& (-11,2.8)}; %Top hole left side in \Sigma_4
% (-19,3)*{}="bl41";
%(-22,3)*{}="bl42";
%"bl41";"bl42" **\crv{~**\dir{.} (-20,4)& (-21,4)}; %Bottom hole right side in \Sigma_4
% (-19.5,3.3)*{}="br41";
%(-21.5,3.3)*{}="br42";
%"br41";"br42" **\crv{~**\dir{.} (-20,2.8)& (-21,2.8)}; %Bottom hole left side in \Sigma_4
(10,18)*{ \Sigma^+};
(10,-19)*{ \Sigma^-};
(-18,0)*{\overline{H}};
(18,0)*{H};
	 \endxy};
	 \endxy}
	 \vspace{2.1cm}
	 	 \caption{\small Pictured above are two possibilities for $Y$. The first has $\Sigma_- = \emptyset$, making $H$ a handlebody. In the second figure, $\Sigma_-$ is connected.}
\end{figure}

Before stating the main result, we mention that the definition of $\CS$ requires a choice of normalization. When $G$ is simple this choice can be made in an essentially unique way. However, for arbitrary compact $G$ the situation is not as simple. It turns out that, in general, this normalization can be fixed by choosing a faithful unitary representation $\rho_0 : G\rightarrow \U(W)$, where $W$ is a finite-dimensional Hermitian vector space. One upshot of this approach is that certain computations reduce to the case where $G$ is a classical group; see Remark \ref{ghhhak}. It is convenient to phrase the main result in terms of a lift $\CS_\afa : \A(P) \rightarrow \bb{R}$ of the Chern-Simons function $\CS$; this lift can be defined by fixing a flat reference connection $\afa \in \A_\fl(P)$. See Section \ref{TheChernSimonsFunctional} for more details.

\begin{theorem}\label{finitechernsimons2} 
Let $G$ be a compact, connected Lie group. There is a positive integer $N_G$ such that if $H$, $P$, $Y$, $\rho_0$ are as above, then all critical values of $\CS_\afa: \A(P) \rightarrow \bb{R}$ are integer multiples of $1/ N_G$. 

The dependence of these critical values on the choice of $\rho_0$ is only up to an overall integer multiple. In particular, if the representation $\rho_0$ has image in $\SU(W) \subset \U(W)$, then all critical values are multiples of $2/N_G$. If $\rho_0$ is the complexification of a faithful orthogonal representation of $G $ (see Remark \ref{ghhhak}), then all critical values are multiples of $4/N_G$.
\end{theorem}

Following Wehrheim \cite{Wenergyidentity}, the integer $N_G$ appearing in Theorem \ref{finitechernsimons2} can be defined explicitly as follows. Consider the integer 
$$n_G\defeq \sup_{G' \leq G} \left\{ \vert \pi_0(C(G')) \vert \right\},$$
where the supremum is over all subgroups of $G$, and $C(G')$ denotes the centralizer in $G$. Then $n_G$ is finite since $G$ is compact. We define $N_G$ to be the least common multiple of $\left\{1, 2, \ldots, n_G \right\}$. Thus $N_G \geq 1$ is an integer depending only on $G$.

The definition of $N_G$ can often be refined if one has certain knowledge about $G$ or $P$. In particular, the proof will show that we can take $N_G = 1$ provided the following hypothesis holds. 
\vspace{-0.3cm}
\begin{adjustwidth}{0pt}{}
\begin{customthmQuant}{Hypothesis 1}\label{H} 
\emph{For each connected component $S\subset \Sigma_-$, the identity component of the gauge group acts trivially on $\A_\fl(P\vert_S)$.}
\end{customthmQuant}
\end{adjustwidth}

\medskip

%\vspace{-0.2cm}
%\begin{align*} 
%\tag{Hypothesis 1} \label{H} \intertext{\emph{For each connected component $S\subset \Sigma_-$, the identity component of the gauge group acts trivially on $\A_\fl(P\vert_S)$.}}
%\end{align*}
%\vspace{-1cm}

 For example, \ref{H} holds trivially when $\Sigma_- $ is empty. When $\Sigma_-$ is non-empty, the hypothesis holds when $G = \SO(3)$ and the restriction of $P$ to each component of $\Sigma_-$ is nontrivial. More generally, this hypothesis is satisfied if $G = \U(r)$ or $\PU(r)$, and the integer $c_1(P)\left[S\right]$ is coprime to $r$ for all connected components $S \subset \Sigma_-$; see \cite{WWfloer}. On the other hand, when $\Sigma_-$ is nonempty and the bundle $P$ is trivial, then \ref{H} is never satisfied due to the trivial connection. That being said, it is perhaps worth mentioning that there are other hypotheses that allow one to replace $N_G$ by 1. For example, an argument by Wehrheim in \cite{Wenergyidentity} can be used in our proof below to show that, when $G = \SU(2)$, one can always replace $N_{\SU(2)}$ by 1 in the statement of Theorem \ref{finitechernsimons2}. We also point out that \ref{H} is not assumed in Theorem \ref{finitechernsimons2}; our primary motivation for introducing this hypothesis is to simplify the discussion at various times.

\medskip

Motivated by the techniques of Dostoglou-Salamon \cite[p.633]{DS3} and Wehrheim \cite{Wenergyidentity}, our strategy for proving Theorem \ref{finitechernsimons2} is to show that all flat connections are gauge equivalent to a connection in a certain canonical form. As a consequence, Theorem \ref{finitechernsimons2} can be viewed as a statement about the connected components of $\A_\fl(P)$. For example, we arrive at the following corollary; see Remark \ref{idremark}.

\begin{corollary}
Let $P \rightarrow Y$ be as in Theorem \ref{finitechernsimons2}. Assume \ref{H} is satisfied and either 
\begin{itemize}
\item $G = \U(r)$ or $\SU(r)$ and $\rho_0$ is the standard representation, or 
\item $G = \PU(r)$ and $\rho_0$ is the adjoint representation. 
\end{itemize}
If $a, a' \in \A_\fl(P)$, then there is a gauge transformation $u$ such that $u^*a$ and $a'$ lie in the same component of $\A_\fl(P)$. Moreover, two flat connections $a, a'$ lie in the same component of $\A_\fl(P)$ if and only if $\CS(a) = \CS(a')$. 
\end{corollary}

Our proof also identifies precisely when flat connections on $P$ exist. To state this, consider the commutator subgroup $\left[G, G \right] \subseteq G$. Then the quotient $P/ \left[G, G\right]$ is a torus bundle over $Y$. For example, if $G$ is semisimple then $P / \left[G, G \right] = Y$, and if $G = \U(r)$ then this quotient is the determinant $U(1)$-bundle. The next result follows from the proof of Proposition \ref{connlemma} below.

\begin{corollary}
Let $P \rightarrow Y$ be as in Theorem \ref{finitechernsimons2}. The space $\A_\fl(P)$ of flat connections is non-empty if and only if (i) the restriction $P / \left[G, G \right] \vert_{\partial H}$ is the trivial bundle, and (ii) for any spherical component $S^2 \subseteq \partial H$, the restriction $P\vert_{S^2}$ is the trivial bundle. 
\end{corollary}

The author's primary interest in Theorem \ref{finitechernsimons2} is due to its implications for the instanton energy values on certain non-compact 4-manifolds; see \cite{DunComp}. These 4-manifolds are those of the form $\bb{R} \times H^\infty$, where 
\begin{equation}\label{hdecomp}
H^\infty \defeq H \cup_{\partial H} \left(\left[0, \infty\right) \times \partial H \right)
\end{equation}
is obtained from a Riemannian 3-manifold $H$ by attaching a cylindrical end on its boundary. Given a principal $G$-bundle $P \rightarrow H$, define $P^\infty \rightarrow H^\infty$ similarly. Then the ``manifold at infinity'' of $\bb{R} \times H^\infty$ is the double of $H$ (see Section \ref{TheEnergyOfInstantons}).

\begin{corollary}\label{cor}
Suppose $G$ is a compact, connected Lie group and $H$ is a compact, oriented 3-manifold with boundary. Let $A$ be any finite-energy instanton on $\bb{R} \times P^\infty \rightarrow \bb{R} \times H^\infty$, with the instanton equation defined using the product metric. Then there is a flat connection $a_\flat$ on $\overline{H} \cup_{\partial H} H$ such that the energy of $A$ is $\CS_\afa(a_\flat)$. 
\end{corollary}

Note that the assumptions on $G$ and $H$ are very general. Corollary \ref{cor} is proved in Section \ref{TheEnergyOfInstantons} using an extension of a standard argument; see \cite{Tau2}, \cite{DS3}, \cite{S}, \cite{Wenergyidentity}, \cite{Wlag}, \cite{N}. See also \cite{Y}, \cite{Etesi} for similar results on instanton energies and characteristic numbers for non-compact manifolds.

\medskip

\noindent {\bfseries Acknowledgments:} The author would like to thank Chris Woodward, Sai Kee Yeung, Sushmita Venugopalan, Dan Freed, and Tom Parker for valuable discussions, comments and suggestions. This work was partially supported by NSF Grant DMS 0739208.

\section{Background}

Given a vector bundle $E \rightarrow X$, we will write $\Omega^\bullet (X, E) \defeq \oplus_k \Omega^k(X, E)$ for the space of differential forms on $X$ with values in $E$. We use the wedge product given by $\mu \wedge \nu = \mu \otimes \nu - \nu \otimes \mu$ for real-valued 1-forms $\mu, \nu$. 

Let $G$ be a compact Lie group, and $\rho_0 : G \rightarrow \U(W)$ the faithful unitary representation from the introduction. Then define a bilinear form $\langle \cdot, \cdot \rangle$ on the Lie algebra $\frak{g}$ by setting
\begin{equation}\label{defofinnerproduct}
\langle \mu, \nu \rangle \defeq  - \frac{1}{2\pi^2}\mathrm{Tr}((\rho_0)_* \mu \cdot (\rho_0)_* \nu), \indent \forall \mu, \nu \in \frak{g},
\end{equation}
where the trace is the one on ${\frak{u}}(W)$. (The normalizing factor $1/2\pi^2$ is chosen so that the quantities (\ref{cw2}) and (\ref{cw3}) below are integers. If $\rho_0$ has image in $\SU(W)$ then the more familiar $1/4\pi^2$ can be used.) Since we have assumed $\rho_0$ is faithful, it follows that $\langle \cdot , \cdot \rangle$ is non-degenerate, and so this defines an $\mathrm{Ad}$-invariant inner product on $\frak{g}$.

Suppose $\pi: P \rightarrow X$ is a principal $G$-bundle over a smooth $n$-manifold $X$; we assume $G$ acts on $P$ on the right. Given a right action $\rho: G \rightarrow \mathrm{Diff}(F)$ of $G$ on a manifold $F$ we will denote the associated bundle by $P \times_G F \defeq (P \times F) / G$. If $F = V$ is a vector space and $G \rightarrow \mathrm{Diff}(V)$ has image in $\GL(V) \subset \mathrm{Diff}(V)$, then $P \times_G V$ is a vector bundle and we will write $P(V) \defeq P \times_G V$. Pullback by $\pi$ induces an injection
$$\pi^*: \Omega^\bullet(X, P(V)) \hookrightarrow \Omega^\bullet(P, P \times V)$$
with image the space of forms that are equivariant and horizontal. 

We will write $P(\frak{g})$ for the \emph{adjoint bundle} associated to the adjoint representation  $G \rightarrow \GL(\frak{g})$. The Lie bracket $\left[ \cdot, \cdot \right]$ on $\frak{g}$ is Ad-invariant, and so this combines with the wedge to define a bilinear map $\mu \otimes \nu \mapsto \left[ \mu \wedge \nu \right]$ on $\Omega^\bullet(X, P(\frak{g}))$, endowing $\Omega^\bullet(X, P(\frak{g}))$ with the structure of a graded algebra. Similarly, the $\mathrm{Ad}$-invariance of the inner product $\langle \cdot, \cdot \rangle$ implies that it induces a fiber-wise inner product on the vector bundle $P(\frak{g})$. This combines with the wedge to give a graded bilinear map
$$\Omega^k(X, P(\frak{g})) \otimes \Omega^l(X, P(\frak{g})) \longrightarrow \Omega^{k+l}(X), \indent \mu \otimes \nu \longmapsto \langle \mu \wedge \nu \rangle.$$

	\subsection{Gauge theory}		

We denote by ${\mathcal{A}}(P) $ the set of all connections on $P$. By definition, $\A(P)$ consists of the elements of $\Omega^1(P, P \times \frak{g})$ that are both $G$-equivariant and vertical. It follows that $\A(P)$ is an affine space modeled on $\pi^* \Omega^1(X, P(\frak{g})) \cong \Omega^1(X, P(\frak{g}))$. We will write $\A^{1}(P)$ for the completion of $\A(P)$ with respect to the $H^1$-Sobolev norm; we will always assume $\A^1(P)$ is equipped with the $H^1$-topology. The space $\A^1(P)$ is independent of choices when $X$ is compact; when $X$ is non-compact the $H^1$-norm depends on the choice of a smooth reference connection at infinity.

Given any representation $\rho: G \rightarrow \GL(V)$, each connection $A \in {\A}(P)$ determines a covariant derivative 
$$d_{A, \rho}: \Omega^\bullet(X, P(V))  \longrightarrow  \Omega^{\bullet + 1} (X, P(V)), \indent \mu  \longmapsto (\pi^*)^{-1}\left( d \left( \pi^*\mu\right)+ \rho_*(A)\wedge \pi^*\mu\right),$$
where $d$ is the trivial connection on $P \times V$. When considering the adjoint representation, we will write $d_{\afA} \defeq d_{\afA, \mathrm{Ad}}$. The \emph{curvature endomorphism} $\mathrm{curv}(d_{A, \rho}) \in \Omega^2(X, \End(P(V)))$ is defined by the relation
			$$d_{A, \rho} \circ d_{A, \rho} \mu = \mathrm{curv}(d_{A, \rho}) \wedge \mu$$ 
			for all $\mu \in \Omega^\bullet(X, P(V))$. We define the \emph{curvature (2-form)} of $A$ by
			$$F_{A} = (\pi^*)^{-1}\left(d A +   \frac{1}{2} \left[A \wedge A \right]\right) \in \Omega^2(X, P(\frak{g})).$$
			The curvature 2-form $F_A$ recovers the curvature endomorphism $\mathrm{curv}(d_{A, \rho})$ in any representation $\rho$ in the sense that 
			\begin{equation}\label{curveq}
			\rho_* F_A = \mathrm{curv}(d_{A, \rho}).
			\end{equation}
			Taking $\rho=\mathrm{Ad}$, we therefore have $			\mathrm{curv}(d_A) \wedge \mu = \left[ F_A\wedge \mu \right]$ for all $\mu \in \Omega^\bullet(X, P(\frak{g}))$. Given any $A \in \A(P)$, the covariant derivative and curvature satisfy 
			$$d_{A+ \mu} = d_A + \left[ \mu \wedge \cdot \right], \indent F_{A + \mu}  =  F_A + d_A \mu + \frac{1}{2} \left[\mu \wedge \mu \right],$$
			for all $\mu \in\Omega^1(X, P(\frak{g}))$. We also have the Bianchi identity $d_A F_A = 0$. A connection $A$ is \emph{flat} if $F_A = 0$, and we denote the set of all smooth (resp. $H^1$) flat connections on $P$ by $\A_{\fl}(P)$ (resp. $\A_\fl^1(P)$).

\medskip

Suppose $X$ is a closed, oriented 4-manifold. Then associated to the fixed representation $\rho_0: G \rightarrow \U(W)$ from the introduction, we obtain a complex vector bundle $P(W)$ equipped with a Hermitian inner product. In particular, this has well-defined Chern classes $c_i \defeq c_i(P(W)) \in H^{2i}(X, \bb{Z})$.  The usual Chern-Weil formula says
$$		\kappa(P) = \kappa(P; \rho_0) \defeq  \left(c_1^2 - 2c_2  \right)\left[X\right] = - \frac{1}{4\pi^2} \intdd{X} \: \mathrm{Tr}\left( \mathrm{curv}(d_{A, \rho_0}) \wedge \mathrm{curv}(d_{A, \rho_0})   \right)  \in \bb{Z},$$
for any connection $A \in \A(P)$; the Bianchi identity shows this is independent of the choice of $A$. Here $\mathrm{Tr}(\mu \wedge \nu)$ is obtained by combining the wedge with the trace on $\frak{u}(W)$. Then equations (\ref{defofinnerproduct}) and (\ref{curveq}) show
		\begin{equation}\label{cw2}
		\kappa(P) = \frac{1}{2} \intdd{X} \: \langle F_A \wedge F_A \rangle.
		\end{equation}
		
\begin{remark}
		This characteristic number can be equivalently defined as follows. Let $B\U(W)$ be the classifying space for the unitary group, and let $\kappa \in H^4(B\U(W), \bb{Z})$ be given by the square of the first Chern class minus two times the second Chern class. Then $\kappa(P) \in H^4(X, \bb{Z}) \cong \bb{Z}$ is obtained by pulling back $\kappa$ under the map $X \rightarrow BG \rightarrow B\U(W)$; here the first arrow is the classifying map for $P$, and the second is induced by the representation $\rho_0: G \rightarrow \U(W)$. 
\end{remark}		
		
		 It follows immediately from the definition that $\kappa(P)$ is even if the mod-2 reduction of $c_1$ vanishes. Now suppose $\rho_0$ is obtained by complexifying a (real) orthogonal representation $ G \rightarrow \mathrm{O}(V)$. Then $P(W) = P(V)_{\bb{C}}$ is the complexification of the real vector bundle $P(V)$ and so $c_1 = 0$ vanishes. If, in addition, $X = S^1 \times Y$ is a product, then a characteristic class argument shows that $c_2$ is even (e.g., see \cite[Section 4.3]{DunPSU}), and so $\kappa(P)$ is a multiple of 4. 
		 
		 For example, consider the case where $G = \SO(r)$ with $r \geq 2$, and $\rho_0= \mathrm{Ad}_{\bb{C}}$ is the complexified adjoint representation. Then $\kappa(P) = 2(r-2)p_1(P(\bb{R}^r))\left[X \right]$, where $p_1(P(\bb{R}^4))$ is the Pontryagin class of the vector bundle associated to the standard representation of $\SO(r)$. 
		 
		 As a second example, consider $G = \SU(r)$. Then the integers $\kappa$ coming from the complexified adjoint and standard representations are related by
		 $$\kappa(P ; \textrm{Ad}_{\bb{C}}) = 2r \: \kappa(P ; \textrm{standard}).$$

		A \emph{gauge transformation} on $P$ is a $G$-equivariant bundle map $ P \rightarrow P$ covering the identity. The set ${\G}(P)$ of gauge transformations on $P$ forms a group, called the \emph{gauge group}. One may equivalently view the gauge group as the set of $G$-equivariant maps $P \rightarrow G$. Here $G$ acts on itself by conjugation of the inverse, making it a \emph{right} action. A third equivalent way to view $\G(P)$ is as the space of sections of the bundle $P \times_G G \rightarrow X$, where $P \times_G G$ is formed using the same action of $G$ on itself.
		
		Denote by $\G_0 = \G_0(P)$ the connected component of the identity in $\G(P)$. We need to specify a topology on $\G(P)$ for the term ``connected component'' to be meaningful, and we do this by viewing $\G(P)$ as a subspace of the space of functions $P \rightarrow G$, equipped with the $H^2$-topology (however, any other H\"{o}lder or Sobolev topology would determine the same connected components). We denote by $\G^2(P)$ the completion of $\G(P)$ in the $H^2$-topology. %Note that this depends on a choice of faithful representation of $G$ (see \cite[Appendix B]{Wuc}), and we take $\rho_0$ for this choice. 

The gauge group acts on $\Omega^\bullet(P, P \times \frak{g})$ and ${\A}(P) \subset \Omega^\bullet(P, P \times \frak{g})$ by pullback. When the dimension of $X$ is three or less, this action is smooth with the specified topologies \cite[Appendix A]{Wuc}. We note that the action of a gauge transformation $u$ on a connection $A$ can be expressed as 
		\begin{equation}\label{gaugeaction}
		u^* A= u^{-1}A u + u^{-1} du,
		\end{equation}
		where the concatenation on the right is matrix multiplication (e.g., defined via the matrix representation $\rho_0$) and $du$ is the linearization of $u: P \rightarrow G$. In dimensions three or less, Equation (\ref{gaugeaction}) combines with the Sobolev multiplication theorem to show that if $u$, $A$ and $u^* A$ are all of Sobolev class $H^1$, then $u$ is actually of Sobolev class $H^2$. 
		
		The group $\G(P)$ also acts on $\Omega^\bullet(X, P(\frak{g}))$ by the pointwise adjoint action $(\xi, u) \mapsto \mathrm{Ad}(u^{-1}) \xi$. In particular, the curvature of $A \in \A(P)$ transforms under $u\in \G(P)$ by
			$$F_{u^*A} = \mathrm{Ad}(u^{-1})F_A.$$

			We introduce a notation convention that is convenient when the dimension of the underlying space $X$ is relevant. If $\dim X = 4$, then we use $A, U$ for connections and gauge transformations; if $\dim X = 3$, then we use $a, u$ for connections and gauge transformations; if $\dim X = 2$, then we use $\alpha, \mu$ for connections and gauge transformations. For example, this provides an effective way to distinguish between a path of gauge transformations $\mu: I \rightarrow \G(P)$ on a surface $X$, and its associated gauge transformation $u \in \G(I \times P)$ on the 3-manifold $I \times X$ defined by $u \vert_{\left\{ t\right\} \times P} = \mu(t)$.

\subsection{The Chern-Simons functional}\label{TheChernSimonsFunctional}

Fix a closed, connected, oriented 3-manifold $Y$, as well as a principal $G$-bundle $P \rightarrow Y$. The space of connections admits a natural 1-form $\lambda \in \Omega^1(\A(P), \bb{R})$ defined at $a \in \A(P)$ by
$$\lambda_{a}: T_a \A(P)  \longrightarrow \bb{R}, \indent v \longmapsto \intd{Y} \: \langle v \wedge F_{a} \rangle.$$
The Bianchi identity shows that this is a closed 1-form. Since $\A(P)$ is contractible it follows that $\lambda$ is exact. Fixing a reference connection $\afa_0$, this exact 1-form can therefore be integrated along paths from $\afa_0$ to obtain a real-valued function $\CS_{a_{0} } : \A(P) \rightarrow \bb{R}$. One can compute that $\CS_{a_{0} }$ is given by the formula
			$$\CS_{a_{0} }(a) \defeq \int_{Y} \langle F_{a_0} \wedge v \rangle +  {\textstyle\frac{1}{2}} \langle d_{a_0} v \wedge v \rangle  + {\textstyle\frac{1}{6}} \langle \left[ v \wedge v \right] \wedge v \rangle,$$
			where we have set $v \defeq  a - a_{0}  \in \Omega^1(Y, P(\frak{g}))$. We will typically choose $a_0$ to be flat, but this is not always convenient. In general, however, changing $a_0$ changes $\CS_{a_0}$ by a constant. Projecting $\CS_{\afa_0}$ to the circle $\bb{R} / \bb{Z}$, one obtains the Chern-Simons function $\CS : \A(P) \rightarrow \bb{R} / \bb{Z}$ from the introduction; we will refer to the lift $\CS_{\afa_0}$ as the \emph{Chern-Simons functional}. Moreover, $\CS_{a_0}$ has a smooth extension from the smooth connections $\A(P)$ to the $H^1$-completion $\A^1(P)$.

			Suppose $a, a' \in \A(P)$. Any path $a(\cdot) : \left[0, 1\right] \rightarrow \A(P)$ from $a$ to $a'$ can be interpreted as a connection $A$ on $\left[0, 1\right] \times P \rightarrow \left[0, 1 \right] \times Y$ by requiring that it restricts to $a(t)$ on $\left\{t \right\} \times Y$. It follows from the definitions that
			$$\begin{array}{rcl}
			\CS_{a_0}(a') - \CS_{a_{0} }(a) & = & \fracd{1}{2} \intd{I \times Y}\: \langle F_A \wedge F_A \rangle.
			\end{array}$$
		In the special case where $a' = \afu^*a$, with $u \in \G(P)$, the connection ${\afA}$ descends to a connection on the mapping torus 
		$$P_{\afu} \defeq I \times P / (0, u(q)) \sim (1, q),$$
		which is a bundle over $S^1 \times Y$. Then the above gives
		\begin{equation}\label{cw3}
		\CS_{a_0}(u^*a) - \CS_{a_0}(a) = \fracd{1}{2} \intd{S^1 \times Y}\: \langle F_A \wedge F_A \rangle =  \kappa(P_u) \in \bb{Z},
		\end{equation}
		where we used (\ref{cw2}) in the second equality. It follows that the value of this depends only on the path component of $u$ in $\G(P)$. Equation (\ref{cw3}) also shows that $\CS_{\afa_0}$ is invariant under the subgroup of gauge transformations $\afu$ with $\kappa(P_{\afu})  = 0$ (the ``degree zero'' gauge transformations), and that the circle-valued function $\CS : \A(P) \rightarrow \bb{R} / \bb{Z}$ is invariant under the full gauge group $\G(P)$.

		 \begin{remark}\label{ghhhak}
			The discussion following equation (\ref{cw2}) shows that if the mod-2 reduction of $c_1(P_u(W))$ vanishes, then (\ref{cw3}) is even. Similarly, if the fixed representation $\rho_0$ is the complexification of a real representation, then (\ref{cw3}) is a multiple of 4. 
		 \end{remark}

For completeness we show that the space of flat connections on $P$ is locally path-connected. This implies, for example, that the Chern-Simons critical values are always isolated since the moduli space $\A_\fl(P) / \G(P)$ is compact and $\CS_{\afa_0}$ is constant on the path components of $\A_\fl(P)$.

		\begin{proposition}\label{propconn}
		The space $\A^1_\fl(P)$ of flat connections is locally path-connected. In particular, the path components are the connected components. 
	\end{proposition}

	\begin{proof}	 
		  R\r{a}de \cite{Rade} used the heat flow associated to the Yang-Mills equations to show that there is some $\eps_P> 0$ such that if $a \in \A^1(P)$ is a connection with $\Vert F_{a} \Vert_{L^2} \leq \eps_P$, then there is a nearby flat connection
				$$\rheat(a) \in \A_\fl^1(P).$$
				R\r{a}de shows that the map $a \mapsto \rheat(a)$ is continuous, gauge equivariant, and restricts to the identity on $\A_\fl^1(P)$.

				Let $a_0, a_1 \in \A_\fl^1(P)$. We want to show that if $a_0$ and $a_1$ are close enough (in $H^1$), then they are connected by a path in $\A_\fl^1(P)$. Consider the straight-line path $a(t) = a_0 + t(a_1 - a_0)$. Then 
				$$F_{a(t)} = t d_{a_0}(a_1 - a_0) + \frac{t^2}{2} \left[ a_1 -a_0 \wedge a_1 - a_0 \right],$$ 
				and so
		$$\Vert F_{a(t)} \Vert_{L^2}  \leq  \Vert d_{a_0} (a_1-a_0) \Vert_{L^2} + \Vert a_1 - a_0 \Vert^2_{L^4} \leq C \left( \Vert a_1 - a_0 \Vert_{H^1} + \Vert a_1 - a_0 \Vert_{H^1}^2\right),$$
		where we have used the Sobolev embedding $H^1 \hookrightarrow L^4$. Then $a(t)$ is in the realm of the R\r{a}de's heat flow map for all $t \in \left[0, 1 \right]$, provided $\Vert a_1 - a_0 \Vert_{H^1} < \min\left\{1, \epsilon_P / 2C \right\}$. When this is the case, $t \mapsto  \rheat(a(t)) \in \A_\fl^1(P)$ is a path from $a_0$ to $a_1$, as desired.

	\end{proof}

\section{Chern-Simons values and instantons}

We prove Theorem \ref{finitechernsimons2} and Corollary \ref{cor} in Sections \ref{Proof1} and \ref{TheEnergyOfInstantons}, respectively. We take a TQFT approach to the proof of Theorem \ref{finitechernsimons2} in the sense that we treat each connection on $Y = \overline{H} \cup_\partial H$ as a pair of connections on $H$ that agree on the boundary. This reduces the problem to a study of the flat connections on $H$ and $\partial H$, which is the content of Section \ref{comp}.

\subsection{The components of the gauge group and the space of flat connections}\label{comp}			
			
			In this section, we fix a principal $G$-bundle $P \rightarrow X$, where $X$ is a manifold with (possibly empty) boundary. The action of the gauge group is rarely free. To account for this, it is convenient to consider the \emph{based gauge group} $\G_p = \G_p(P)$ defined as the kernel of the map $\G(P) \rightarrow G$ given by evaluating $u: P \rightarrow G$ at some fixed point $p \in P$. If $X$ is connected, then $\G_p$ acts freely on $\A(P)$ (in general, the stabilizer in $\G(P)$ of a connection $A \in \A(P)$ can be identified with the image in $G$ of the evaluation map $u \to u(p)$). 
			
			Let $\widetilde{G} \rightarrow G$ be the universal cover. We will be interested in the subgroup ${\mathcal{H}} = {\mathcal{H}}(P)$ of gauge transformations $u: P \rightarrow G$ that lift to $G$-equivariant maps $\smash{\widetilde{u}: P \rightarrow \widetilde{G}}$, where the (right) action of $G$ on $\smash{\widetilde{G}}$ is induced by the conjugation action of $\smash{\widetilde{G}}$ on itself. 
			
			\begin{lemma}\label{44}
			The subgroup ${\mathcal{H}}$ is a union of connected components of $\G(P)$. In particular, ${\mathcal{H}}$ contains the identity component $\G_0$ of $\G(P)$.
			\end{lemma}
			
			\begin{proof}
			 Consider the aforementioned right action of $G$ on $\smash{\widetilde{G}}$. Use this action to define a bundle $\smash{P \times_G \widetilde{G}} \rightarrow X$, and consider the natural projection $\smash{P \times_G \widetilde{G} \rightarrow P \times_G G}$. Viewing a gauge transformation $u$ as a section of $P \times_G G \rightarrow X$, the defining condition of ${\mathcal{H}}$ is equivalent to the existence of a section $\smash{\widetilde{u}: X \rightarrow P \times_G \widetilde{G}}$ lifting $u$. It follows from the homotopy lifting property for the covering space $\smash{P \times_G \widetilde{G} \rightarrow P \times_G G}$ that if $u$ can be connected by a path to an element of ${\mathcal{H}}$, then $u \in {\mathcal{H}}$. 
			 \end{proof}
			
\begin{lemma}\label{2}
Suppose $G$ is compact and connected, and that $X$ has the homotopy type of a connected 2-dimensional CW complex. Then ${\mathcal{H}} \cap \G_p$ is connected, and the inclusion $\G_p \subseteq \G(P)$ induces a bijection $\pi_0(\G_p) \cong \pi_0(\G(P))$. Consequently, ${\mathcal{H}}$ is the identity component $\G_0$ of $\G(P)$.
\end{lemma}

\begin{proof}
First we show that ${\mathcal{H}} \cap \G_p$ is connected. For $u \in {\mathcal{H}}$, let $\widetilde{u}$ be a lift as above. Note that if $u \in \G_p$, then $\smash{\widetilde{u}(p) \in Z(\widetilde{G})}$ is in the center and so $\smash{\widetilde{u}(p)^{-1} \widetilde{u}}$ is another equivariant lift of $u$. In particular, by replacing $\widetilde{u}$ with $\smash{\widetilde{u}(p)^{-1} \widetilde{u}}$, we may assume $\widetilde{u}$ has been chosen so that $\widetilde{u}(p) = e \in\widetilde{G}$. Moreover, by homotoping $u$ we may assume that $u$ (hence $\widetilde{u}$) restricts to the identity on $\pi^{-1}(B)$, where $B \subset X$ is some open coordinate ball around $x = \pi(p)$. The topological assumptions imply that $B$ can be chosen so the complement $X - B$ deformation retracts to its 1-skeleton. Since $G$ is connected, the restriction $P \vert_{X - B} \rightarrow X-B$ is trivializable. By equivariance, we may therefore view $\widetilde{u}$ simply as a map
$$\widetilde{u}: (X - B , \partial B) \longrightarrow (\widetilde{G}, e);$$
that is, we may view $\widetilde{u}$ as a map on the base with no equivariance restrictions. Since $G$ is compact and connected, it follows that $\widetilde{G}$ is 2-connected. Up to homotopy, $(X - B , \partial B)$ is a 2-dimensional CW pair, so $\widetilde{u}$ can be homotoped \emph{rel} $\partial B$ to the identity, which shows ${\mathcal{H}} \cap \G_p$ is connected. %See Hatcher Lemma 4.6

Now we show $\pi_0(\G_p) \cong \pi_0(\G(P))$. We may homotope any gauge transformation $u: P \rightarrow G$ so that it is constant on $\smash{\pi^{-1}(B) \subset P}$, with $B$ as above. Just as above $\smash{P \vert_{\overline{B}} \rightarrow \overline{B}}$ is the trivial bundle, so gauge transformations on $\smash{P \vert_{\overline{B}}}$ are exactly maps $\overline{B} \rightarrow G$. Since $G$ is connected, we can obviously find a homotopy $rel\; \partial B$ of $u: (\overline{B}, \partial B) \rightarrow (G, u(p))$ to a map that sends $x \in B$ to the identity. This shows that $u$ can be homotoped to an element of $\G_p$.

Finally, by Lemma \ref{44} we have $\G_0 \subseteq {\mathcal{H}}$, while the reverse inclusion follows from the conclusions of the previous two paragraphs.

	\end{proof}

			Fix $x \in X$ as well as a point $p \in P$ over $x$. It is well-known that the holonomy provides a map $\mathrm{hol}: \A_\fl (P)  \rightarrow \hom(\pi_1(X, x), G)$. This intertwines the action of $\G(P)$ on $\A_\fl(P)$ with the conjugation action of $G$ on itself in the sense that if $\gamma: (S^1, 1) \rightarrow (X, x)$ is a smooth loop, then
			$$\mathrm{hol}_{u^* \afA}(\gamma) = u(p)^{-1} 	\mathrm{hol}_{\afA}(\gamma) 	u(p)$$
			for all gauge transformations $u \in \G(P)$ and flat connections $\afA$; see \cite[Prop. 4.1]{KN} and \cite{AB}. Moreover, the holonomy descends to a topological embedding
			$$\A_\fl(P) / \G_p \hookrightarrow \hom(\pi_1(X, x), G)$$
			with image a union of connected components that are determined by the topological type of the bundle $P$. To determine this set of image components for a given bundle $P$ it is useful to consider the following variation dating back to Atiyah and Bott \cite{AB}. Let $j: G \rightarrow P$ denote the embedding $g \mapsto p \cdot g^{-1}$ (recall $G$ acts on $P$ on the right), and let $j_*$ denote the induced map on $\pi_1$. Consider the universal cover $\smash{\widetilde{G} \rightarrow G}$ and denote by $\iota: \pi_1(G) \hookrightarrow \smash{Z(\widetilde{G})}$ the natural inclusion into the center of $\smash{\widetilde{G}}$. Then there is a homeomorphism
				\begin{equation}\label{1}
			\A_\fl(P) / ({\mathcal{H}} \cap \G_p) \cong \left\{ \rho \in \hom(\left. \pi_1(P, p), \widetilde{G}) \: \right| \: \rho \circ j_* = \iota  \right\}.
			\end{equation}
			We defer a proof of (\ref{1}) until the end of this section.

			\begin{proposition}\label{connlemma}
			Assume $G$ is compact and connected. Suppose $X$ is either a closed, connected, oriented surface, or $X = H$ is a compression body. Then the space of flat connections $\A_\fl(P)$ is connected when it is non-empty.	
			\end{proposition}

			\begin{proof}
			By Lemma \ref{2} the group $\mathcal{H} \cap \G_p = \G_0 \cap \G_{p}$ is connected. Moreover, it acts freely on $\A_\fl(P)$ since this is the case with $\G_p$. We will show the space on the right-hand side of (\ref{1}) is connected. The proposition follows immediately by the homotopy exact sequence for the bundle $\A_\fl(P)  \rightarrow \A_\fl(P) / ({\mathcal{H}} \cap \G_p)$.
			
			First assume $X $ is a surface of genus $g \geq 0$. For $g =0$, the space $\A_\fl(P) /(\G_0 \cap \G_p)$ is either a single point or empty, depending on whether $P$ is trivial or not. We may therefore assume $g \geq 1$. The bundle $P \rightarrow X$ is determined up to bundle isomorphism by some $\delta \in \pi_1(G) \subset Z(\widetilde{G})$. Since $\widetilde{G}$ is simply-connected, it follows that $\widetilde{G} = G_1 \times \ldots \times G_k \times \mathbb{R}^l$ for some simple, connected, simply-connected Lie groups $G_1 , \ldots, G_k$. Write $\delta = (\delta_1, \ldots, \delta_k, r)$ according to this decomposition. 
			
			Now we compute $\pi_1(P, p)$. Let $U$ be the complement in $X$ of a point $y$, and let $V$ be a small disk around $y$. Applying the Seifert-van Kampen theorem to the sets $P\vert_U, P \vert_V \subset P$, one finds a presentation for $\pi_1(P, p)$ that consists of generators and relations coming from $\pi_1(G)$, as well as additional generators $\alpha_1, \beta_1, \dots, \alpha_g, \beta_g $ subject to the relation
			\begin{equation}\label{quantity}
			\Pi_{j=1}^g \Big[ \alpha_j, \beta_j \Big]  = \delta,
			\end{equation}		
			 as well as further relations asserting that each element of $\left\{\alpha_i, \beta_i\right\}_i$ commutes with each generator coming from $\pi_1(G)$. Alternatively, the relation (\ref{quantity}) can be viewed as arising when one compares trivializations of $P \vert_U$ and $P \vert_V$ on the overlap $U \cap V$. It follows that $\A_\fl(P) /\G_0 \cap \G_p$ can be identified with the set of tuples $\left(A_{ij}, B_{ij} \right)_{i, j}$, for $1 \leq i \leq k +1$ and $1 \leq j \leq g$, where
			 \begin{itemize}
			 \item[(i)] $A_{ij}, B_{ij} \in G_{i}$, and $\Pi_{j = 1}^g \Big[ A_{ij}, B_{ij} \Big] = \delta_i$ for $1 \leq i \leq k$;
			 \item[(ii)] $A_{kj}, B_{kj} \in \bb{R}^l$, and $\Pi_{j = 1}^g \Big[ A_{kj}, B_{kj} \Big] = r$.
			 \end{itemize}
Since $\bb{R}^l$ is abelian, the tuples $(A_{kj}, B_{kj})_j$ appearing in (ii) can only exist if $r = 0$. This shows that $\A_\fl(P)$ is empty if $r \neq 0$, so we may assume $r = 0$. (Note that $r = 0$ if and only if the torus bundle $P / \left[G, G \right]$, from the introduction, is the trivial bundle.)
			
			For $1 \leq i \leq k$, given any $\delta_i \in \widetilde{G}$ it can be shown that (a) there always exist tuples $(A_{ij}, B_{ij})_j \subset G^{2g}_i$ satisfying $\Pi_{j = 1}^g \left[A_{ij}, B_{ij} \right] = \delta_i$, and (b) the set of such $(A_{ij}, B_{ij})_j$ is always connected; see \cite{AMM}, \cite[Section 2.1]{RSW} or \cite[Fact 3]{HL}. It follows that $\A_\fl(P) /\G_0 \cap \G_p$ is a product of connected spaces and is therefore connected. This finishes the proof in the case where $X$ is a surface.

			Now suppose $X = H$ is a compression body. Then there is a homotopy equivalence $H \simeq \left( \bigvee_{i = 1}^s \Sigma_i \right) \vee ( \bigvee_{i = 1}^t S^1) $ onto a wedge sum of closed, connected, oriented surfaces $\Sigma_i$ and circles; note that the surfaces can be identified with the components of the incoming end $\Sigma_- \subset \partial H$. It follows from (\ref{1}) that $\A_\fl(P) /\G_0 \cap \G_p$ is homeomorphic to 
$$\left\{ \left. \rho \in \hom(\pi_1(P_1), \widetilde{G})  \: \right| \: \rho \circ j_* = \iota \right\} \times \ldots \times \left\{ \left. \rho \in \hom(\pi_1(P_s), \widetilde{G})  \: \right| \: \rho \circ j_* = \iota \right\} \times \left( \widetilde{G}\right)^t,$$
		where $P_i \rightarrow \Sigma_i$ is the restriction of $P$ to the surface $\Sigma_i \subset H$. By the previous paragraph this is a product of connected spaces, and so is itself connected. 
			\end{proof}
			
			\begin{remark}\label{544}
			The above proof shows that, when $H$ is a compression body, restricting to the incoming end $\Sigma_- \subset \partial H$ yields a surjective map
			$$\frac{\A_\fl(P)}{ \G_0 \cap \G_p }\longrightarrow \frac{\A_\fl(P_1)}{\G_0(P_1) \cap \G_{p_1}(P_1) }\times \ldots \times \frac{\A_\fl(P_s)}{ \G_0(P_s) \cap \G_{p_s}(P_s)}$$
			that is a (trivial) principal $\widetilde{G}^t$-bundle. Similarly, restricting to the outgoing end $\Sigma_+ \subset \partial H$ yields an injection
			$$\frac{\A_\fl(P) }{\G_0 \cap \G_p }\hookrightarrow \frac{\A_\fl(P_+)}{ \G_0(P_+) \cap \G_{p_+}(P_+)},$$
			where $P_+ \rightarrow \Sigma_+$ is the restriction of $P$. In particular, a flat connection on $P \rightarrow H$ is determined uniquely, up to $\G_0(P) \cap \G_p(P)$, by its value on the boundary component $\Sigma_+$, and hence by its value on $\partial H$. 
			\end{remark}
			
			Now we verify (\ref{1}). This can be viewed as arising from the \emph{$\widetilde{G}$-valued holonomy}, which we now describe. Let $A \in \A(P)$ be a connection. Given a smooth loop $\gamma: S^1 = \bb{R} / \bb{Z} \rightarrow P$, consider the induced loop in the base $\pi \circ \gamma: S^1 \rightarrow X$. Use this to pull $P$ back to a bundle over the circle $(\pi \circ \gamma)^*P \rightarrow S^1$. The standard ($G$-valued) holonomy determines a lift $\mathrm{hol}_{A}(\pi \circ \gamma)$ of the quotient map $\left[0, 1 \right] \hookrightarrow S^1 = \bb{R} / \bb{Z}$:
			\begin{equation} \label{standardholon}
			 \begin{diagram}
			 	& & (\pi \circ \gamma)^* P \\
			 	 & \ruDashto^{\mathrm{hol}_{A}(\pi \circ \gamma)} & \dTo \\
			  \left[ 0, 1 \right] & \rInto & S^1 \\
			\end{diagram}
			\end{equation}
			 and this lift is unique if we require that it sends $0$ to $\gamma(0) \in (\pi \circ \gamma)^* P$. On the other hand, $\gamma$ determines a trivialization of this pullback bundle
			 $$(\pi \circ \gamma)^* P \cong S^1 \times G, \indent \gamma(t) \longmapsto (t, e).$$
			 Compose the lift in (\ref{standardholon}) with this isomorphism and then with the projection to the $G$-factor in $S^1 \times G$ to get a map%To verify the condition $\rho \circ j_* = \iota$, let $g_t$ represent an element of $\pi_1(G)$, and let $\gamma = j \circ g_t$, which is the map $t \mapsto p \cdot g_t^{-1}$. Then $\pi \circ \gamma$ is constant and so the lift $\mathrm{hol}_A(\pi \circ \gamma)$ is constantly $\gamma(0) = p$. Then $\gamma(0) = \gamma(t) \cdot g_t$. The trivialization $\gamma(t) \longmapsto (t, e)$ is $G$-equivariant, and so it sends the constant map $\mathrm{hol}_A(\pi \circ \gamma(t)) = \gamma(0)$ to $(t, g_t)$, which projects to the give the loop $t \rightarrow g_t$ in $G$. Essentially by definition, thie lifts to the relevant element of $\widetilde{G}$, and so $\rho \circ j = \iota$.
			 \begin{equation}\label{hol223}
			 \mathrm{hol}_{A}(\pi \circ \gamma): \left[0, 1\right] \longrightarrow G,
			 \end{equation}
			 which we denote by the same symbol we used for the standard holonomy. Then (\ref{hol223}) sends $0$ to the identity $e \in G$ and the value at 1 recovers the standard holonomy for $A$ around $\gamma$. Viewing $\widetilde{G} \rightarrow G$ as a covering space, $\mathrm{hol}_{A}(\pi \circ \gamma)$ lifts to a unique map $\widetilde{\mathrm{hol}}_{A}(\pi \circ \gamma): \left[0, 1\right] \rightarrow \widetilde{G}$ that sends $0$ to $e$. Then we declare the \emph{$\widetilde{G}$-valued holonomy} of $A$ around $\gamma$ to be the value at 1:
			 $$\mathrm{hol}^{\widetilde{G}}_{A}(\gamma) \defeq  \widetilde{\mathrm{hol}}_{A}(\pi \circ \gamma)(1) \in \widetilde{G}.$$
			 As with the standard holonomy, one can check that this is multiplicative under concatenation of paths $\gamma$. Similarly, this is equivariant in the following sense. Suppose $u \in {\mathcal{H}}$ and so $u$ lifts to a $G$-equivariant map $\widetilde{u}: P \rightarrow \widetilde{G}$. Setting $g \defeq \widetilde{u}(p)$, we have
			$$ \mathrm{hol}^{\widetilde{G}}_{u^*A}(\gamma) = g^{-1} \mathrm{hol}^{\widetilde{G}}_{A}(\gamma)g.$$

			 Next, suppose $A$ is a flat connection. Then $\mathrm{hol}^{\widetilde{G}}_{A}(\gamma) $ depends only on the homotopy class of $\gamma$. It follows from the above observations that the $\smash{\widetilde{G}}$-valued holonomy defines a map $\smash{\A_\fl(P) \rightarrow \hom(\pi_1(P, p), \widetilde{G})}$, and this intertwines the actions of ${\mathcal{H}}$ and $\smash{\widetilde{G}}$. Moreover, from the definition of $\G_p$ we have that the $\widetilde{G}$-valued holonomy is invariant under the action of ${\mathcal{H}} \cap \G_p$. We therefore have a well-defined map $\smash{\A_\fl(P) / {\mathcal{H}} \cap \G_p \rightarrow \hom(\pi_1(P), \widetilde{G})}$. It follows from the definitions above that the image lies in the right-hand side of (\ref{1}). That this map is a homeomorphism follows from the analogous argument for the standard holonomy, together with the commutativity of the following diagram.
			  $$\begin{diagram}
			 	\A_\fl(P) / ({\mathcal{H}} \cap \G_p) & \rTo  & \left\{ \left.\rho \in \hom(\pi_1(P), \widetilde{G})  \: \right| \: \rho \circ j_* = \iota \right\}\\
			 	 \dTo & & \dTo \\
			    \A_\fl(P) / \G(P) & \rInto & \hom(\pi_1(X), G) / G
			\end{diagram}$$

\subsection{Proof of Theorem \ref{finitechernsimons2}}\label{Proof1}

Write $Y = \overline{H} \cup_{\partial H} H$, where $H$ is a compression body. Fix a collar neighborhood $\left[0, \eps \right) \times \partial H \hookrightarrow H$ for $\partial H$, and use this to define the smooth structure on $Y$. This smooth structure is independent, up to diffeomorphism, of the choice of collar neighborhood, see \cite[Theorem 1.4]{Mil}. The product structure of this collar neighborhood can be used to define a vector field $\nu$ on $Y$ that is normal to $\partial H$ and that does not vanish at $\partial H$. Moreover, we assume $\nu$ is supported near $\partial H$, and so $\nu$ lifts to an equivariant vector field on $P$ that we denote by the same symbol. 

Restriction to each of the $H$ factors in $Y = \overline{H} \cup_{\partial H} H$ determines an embedding
$$\A_\fl^1(P)  \hookrightarrow   \left\{ \left(b ,  c \right)  \in  \A_\fl^1\left(P\vert_H\right) \times \A_\fl^1\left(P\vert_H\right) \left|\begin{array}{rcl}
b \vert_{\partial H} & = &c\vert_{\partial H},\\
- \iota_\nu b \vert_{\partial H} & = &  \iota_\nu c\vert_{\partial H} 
\end{array}\right.\right\}\\$$
given by
\begin{equation}\label{15}
a  \mapsto  (a \vert_{\overline{H}}, \; a \vert_H ).
\end{equation}
A few comments about the defining conditions in the codomain are in order: (i) we are treating $\smash{\nu = \nu \vert_H}$ as a vector field on $H$, viewed as the second factor in $\overline{H} \cup_{\partial H} H$; (ii) the negative sign is due to the reversed orientation of the first factor; and (iii) restriction to the hypersurface $\partial H \subset Y$ extends to a bounded linear map $H^1(Y) \rightarrow L^2(\partial H)$ (see, e.g., \cite[Thm 6.3]{Adams}), and so these equalities should be treated as equalities in the $L^2$ sense. 

Suppose $(b, c)$ is in the codomain of (\ref{15}). These define a connection $a$ on $Y$ by setting $\smash{a \vert_{\overline{H}} = b}$ and $\smash{a \vert_H = b}$. It is straight-forward to check that if $b$ and $c$ are both smooth, then $a$ continuous and of Sobolev class $\smash{H^1}$ on $Y$. Since the smooth connections are dense in $\A^1$, it follows that (\ref{15}) is surjective, and so we may treat (\ref{15}) as an identification.

The bijection (\ref{15}) singles out a preferred subspace that we call the \emph{diagonal}
\begin{equation}\label{diagon}
\left\{\left. (b, b) \in \A^1_\fl\left(P\vert_H\right) \times \A^1_\fl\left(P\vert_H\right) \; \right| \; \iota_\nu b \vert_{\partial H} = 0 \right\} \subset \A^1_\fl(P).
\end{equation}
It is convenient to consider a slightly larger space ${\mathcal{C}} \subset \A^1_\fl(P)$ defined to be the set of flat connections that can be connected by a path to an element of the diagonal (\ref{diagon}).

\medskip
\noindent \emph{Claim: The diagonal (\ref{diagon}) is path-connected. In particular, ${\mathcal{C}}$ is also path-connected.}
\medskip

To see this, consider diagonal elements $(b_0, b_0), (b_1, b_1)$. It suffices to prove the claim under the assumption that $b_0, b_1$ are both smooth and satisfy 
\begin{equation}\label{satisfy}
\iota_\nu b_0 \vert_U = \iota_\nu b_1 \vert_U = 0
\end{equation}
 on some neighborhood $U$ of $\partial H$ (this is because the $H^1$-completion of the space of these connections recovers (\ref{diagon}) and the path-components are stable under completion). By Proposition \ref{connlemma} there is a path of flat connections $t \mapsto b_t \in \A_\fl(P\vert_H)$ connecting $b_0$ and $b_1$. We will be done if we can ensure that $\iota_\nu b_t \vert_{\partial H} = 0$ for all $t \in \left[0, 1\right]$. We will accomplish this by putting $b_t$ in a suitable ``$\nu$-temporal gauge'', as follows. Restrict attention to the bicollar neighborhood $\left(-\eps, \eps \right) \times \partial H \subset Y$ obtained by doubling the collar neighborhood from the beginning of this section. Let $s$ denote the variable in the $\left(-\eps, \eps \right)$-direction and fix a bump function $\beta$ for $U$ that is equal to 1 on $\partial H$. For each $t \in \left[0, 1\right]$, define a gauge transformation $u_t$ at $(s, h ) \in \left(-\eps, \eps \right) \times \partial H $ by the formula
$$u_t(s, h)  \defeq  \exp\left(- \intdd{0}{s} \iota_{\beta \nu(\sigma, h)} b_t(\sigma, h) \: d\sigma\right).$$
Then $u_t$ depends smoothly on all variables, and a computation shows
$$\iota_{\beta \nu} (u_t^*b_t) = 0.$$
Moreover, it follows from (\ref{satisfy}) that $u_t$ is the identity gauge transformation when $t = 0, 1$. The claim follows by extending $u_t$ to all of $Y$ using a bump function. 

\medskip

It follows from the claim that the Chern-Simons functional is constant on ${\mathcal{C}}$, since $\CS_{a_0}$ is \emph{locally} constant on its critical set $\A_\fl^1(P)$. Suppose \ref{H} holds. We will show that every flat connection in $\A^1_\fl (P)$ is gauge equivalent to one in ${\mathcal{C}}$; Theorem \ref{finitechernsimons2} will then follow immediately from Remark \ref{ghhhak}. In fact, by another density argument, it suffices to show that every \emph{smooth} flat connection is gauge equivalent to one in ${\mathcal{C}}$. So we fix $a \in \A_\fl(P)$. As in the proof of the claim, by applying a suitable gauge transformation, we may assume that $\iota_\nu a = 0$. Use (\ref{15}) to identify $a$ with a pair $(b, c) \in \A_\fl\left(P\vert_H\right) \times \A_\fl\left(P\vert_H\right)$. Then $b, c$ agree on the boundary, so by Remark \ref{544}, there is some gauge transformation $u \in \G_0(P \vert_H) \cap \G_p(P \vert_H)$ for which $u^* c = b$. Here we have chosen $p \in H$ to lie in $\Sigma_+ \subset \partial H$, and we are thinking of the $H$ that appears here as the second factor in $Y = \overline{H} \cup_{\partial H} H$. Our immediate goal is to show that $u$ restricts to the identity gauge transformation on the boundary $\partial H = \Sigma_+ \cup \Sigma_-$. Since $p \in \Sigma_+$, it follows that the restriction $u \vert_{\Sigma_+}$ lies in $\G_p(P \vert_{\Sigma_+})$, which acts freely. Since $b$ and $c$ agree on $\Sigma_+$, it must be the case that $u \vert_{\Sigma_+}$ is the identity. Turning attention to $\Sigma_-$, for each component $\Sigma' \subset \Sigma_-$, the restriction $u \vert_{\Sigma'}$ lies in the identity component of the gauge group. In particular, by \ref{H} we have $u \vert_{\partial H} = e$ is the identity. At this point we have that $u$ is a gauge transformation on $H \subset Y$ that is the identity on all of $\partial H$. Then $u$ extends over $\overline{H} \subset Y$ by the identity to define a continuous gauge transformation $\smash{u^{(1)} = (e, u)}$ on $P$. This is of Sobolev class $H^1$. We also have $\smash{(u^{(1)})^*a} \in {\mathcal{C}}$, since under (\ref{15}) the connection $\smash{(u^{(1)})^*a} $ corresponds to the pair $(b, b) = (b, u^*c)$ and we have assumed $\iota_\nu a = 0$. Finally, since $u^{(1)}$, $a$ and $\smash{(u^{(1)})^*a}$ are all $H^1$, it follows from (\ref{gaugeaction}) that $\smash{u^{(1)}}$ is $H^2$. This finishes the proof of Theorem \ref{finitechernsimons2} under \ref{H}.

\begin{remark}\label{idremark}
Continue to assume \ref{H}, and suppose $a, a'$ are flat connections. Then the construction of the previous paragraph shows that there is a gauge transformation $w \in {\mathcal{H}}(P)$ such that $w^*a$ and $a'$ lie in the same path component. If we further assume that $\CS_{a_0}(a ) = \CS_{a_0}(a')$, then it follows that $\kappa(P_w) = 0$. In many cases, if $w \in {\mathcal{H}}$ and $\kappa(P_w) = 0$, then $w$ necessarily lies in the identity component. For example, this is well-known when $G = \U(r)$ or $\SU(r)$ and $\rho_0: G \rightarrow \U(\bb{C}^r)$ is the standard representation \cite[p.79]{FU}, or if $G = \PU(r)$ and $\rho_0$ is the adjoint representation \cite{DunPSU}. In such cases, it follows that $a$ and $a'$ lie in the same component of $\A_\fl(P)$.
\end{remark}

To prove the theorem without \ref{H}, we follow a strategy of Wehrheim \cite{Wenergyidentity}. Let $n_G$ be as in the definition of $N_G$. Without \ref{H} it may not be the case that $u \in \G(P\vert_H)$ restricts to the identity on $\Sigma_-$. Write $\Sigma_- = \Sigma_1 \cup \ldots \cup \Sigma_s$ in terms of its connected components and write $P_i$ for the restriction of $P$ to $\Sigma_i \subset \partial H$. Since $G$ is compact, the stabilizer subgroup in $\G(P_i)$ of each restriction $\smash{b \vert_{\Sigma_j}}$ has only finitely many components, and so there is some integer $n \leq n_G $ for which $u^n\vert_{\Sigma_i}$ lies in the identity component of the stabilizer group for $\smash{b \vert_{\Sigma_j}}$. For simplicity we assume $\smash{u^n\vert_{\Sigma_j} = e}$ is the identity for each $j$; one can check that the following argument can be easily reduced to this case. 

View $H$ as a cobordism from $\Sigma_-$ to $\Sigma_+$ (we may assume $\Sigma_-$ is not empty, otherwise \ref{H} is satisfied), and define a manifold $\smash{Y^{(n)}}$ by gluing $H$ to itself $2n$ times:
\begin{equation}\label{htown}
\overline{H} \cup_{\Sigma_+} H \cup_{\Sigma_-} \overline{H} \cup_{\Sigma_+} \ldots \cup_{\Sigma_+} H \cup_{\Sigma_-};
\end{equation}
this is cyclic in the sense that the $H$ on the right is glued to the $\overline{H}$ on the left along the boundary component $\Sigma_-$. Define a bundle $\smash{P^{(n)} \rightarrow Y^{(n)}}$ similarly. Then $a = (b, c)$ determines a continuous flat connection on $\smash{P^{(n)}}$ by the formula
$$a^{(n)} \defeq (b, c, b, c, \ldots, b, c);$$
the notation means that the $k$th component lies in the $k$th copy of ${H}$ in (\ref{htown}). Similarly, the reference connection $a_0$ defines a reference connection $\smash{a_0^{(n)}}$ on $\smash{P^{(n)}}$, and the gauge transformation $u$ determines a continuous gauge transformation on $\smash{P^{(n)}}$ by
$$u^{(n)} \defeq (e, u, u, u^2, u^2, \ldots, u^{n-1}, u^{n-1}, u^n).$$
Let $\CS^{(n)}$ denote the Chern-Simons functional for $P^{(n)}$ defined using $a^{(n)}_0$. Then (\ref{cw3}) and the additivity of the integral over its domain give
$$\CS^{(n)}((u^{(n)})^*a^{(n)}) = \CS^{(n)} (a^{(n)})  +  \kappa( P_{u^{(n)}} ) = n\CS_{a_0}(a) + \kappa( P_{u^{(n)}} ).$$
On the other hand, the pullback of $\smash{a^{(n)}}$ by $\smash{u^{(n)}}$ is $\smash{(b, b, u^*b, u^* b, \ldots, (u^{n-1})^* b, (u^{n-1})^* b)}$, and so
$$\CS^{(n)}((u^{(n)})^*a^{(n)})  = n \CS_{a_0}(a') +  k_n, \indent k_n\defeq \frac{1}{2}n(n-1) \kappa(P_u) \in \bb{Z},$$
where $a' \in {\mathcal{C}}$ is the connection corresponding to $(b, b)$ under (\ref{15}). Combining these gives $\CS_{a_0}(a) - \CS_{a_0}(a') \in \frac{1}{n} \bb{Z} \subseteq \frac{1}{N_G} \bb{Z}$.

\subsection{The energies of instantons}\label{TheEnergyOfInstantons}

Let $P^\infty \rightarrow H^\infty$ be as in the statement of Corollary \ref{cor}, and let $g$ be the cylindrical end metric on $H^\infty$. Equip the 4-manifold $\bb{R} \times H^\infty$ with the product metric, and denote by $Q \rightarrow \bb{R} \times H^\infty$ the pullback of $P^\infty$ under the projection $\bb{R} \times H^\infty \rightarrow H^\infty$. The \emph{energy} of a connection $A \in \A(Q)$ is defined to be 
$$\frac{1}{2} \Vert F_A \Vert_{L^2(\bb{R} \times H^\infty)} = \frac{1}{2} \intd{\bb{R} \times H^\infty} \langle F_A \wedge * F_A \rangle,$$
where $*$ is the Hodge star coming from the metric. We will always assume the energy of $A$ is finite. We say that $A$ is an \emph{instanton} if $*F_A = \pm F_A$. It follows that the energy of any instanton is given, up to a sign, by 
\begin{equation}\label{ener}
 \frac{1}{2} \intd{\bb{R} \times H^\infty} \langle F_A \wedge F_A \rangle.
\end{equation}
In this section we will prove Corollary \ref{cor} by showing that (\ref{ener}) is equal to $\CS_{a_0}(a_\flat)$ for some flat connections $a_\flat, a_0$ on $Y \defeq \overline{H} \cup_{\partial H} H$. First we introduce some notation.

\medskip

Recalling the decomposition (\ref{hdecomp}), there is a projection 
\begin{equation}\label{projmap}
\bb{R} \times H^\infty \longrightarrow \bb{H}
\end{equation}
to the upper half-plane, sending $\left\{s \right\} \times H$ to $(s, 0) \in  \bb{H}$, and sending each element of $\left\{(s, t) \right\} \times \partial H$ to $(s, t)$. (This projection is continuous, but \emph{not} differentiable.) Note that for each $\tau  \in (0, \infty)$, the inverse image under (\ref{projmap}) of the semi-circle 
$$\left\{ (\tau \cos(\theta), \tau \sin(\theta))\: \vert\: \theta \in \left[0, \pi \right] \right\} \subset \bb{H}$$ 
is the closed 3-manifold 			
					$$Y_{\tau} \defeq \overline{H} \cup_{\left\{0 \right\} \times \partial H} \left( \left[0, \tau \pi \right] \times \partial H \right) \cup_{\left\{\tau \pi \right\} \times \partial H} H.$$
					In the degenerate case $\tau = 0$, we declare $Y_0$ to be the inverse image under (\ref{projmap}) of the origin; so $Y_0 = \left\{0 \right\} \times H$. Then we have
					$$\bb{R} \times H^\infty = \cup_{\tau \geq 0} Y_\tau.$$					
					Moreover, for each $\tau> 0$, there is an identification $Y_{\tau} \cong Y_1$ induced from the obvious linear map $\left[0, \tau \pi \right] \cong \left[0, \pi \right]$. This identification is continuous, but when $\tau \neq 1$ this identification is not smooth due to the directions transverse to $\left\{0, \tau \pi \right\} \times \partial H$ in $Y_\tau$. We note also that we can identify $Y_1$ with the double $Y$; however we find it convenient to work with $Y_1$ rather than $Y$ at this stage. In summary, we have defined a continuous embedding
					$$\Pi: (0, \infty) \times Y_1 \longrightarrow \bb{R} \times H^\infty$$					
					with image the complement of $Y_0$; this map is not smooth. We think of $\Pi$ as providing certain ``polar coordinates'' on $\bb{R} \times H^\infty$. 
					
					Fix a connection ${A}$. Then we can write the pullback under $\Pi$ as
					$$\Pi^*A = a(\tau) + p(\tau) \: d\tau,$$
					where $\tau$ is the coordinate on $(0, \infty)$, $a(\cdot)$ is a path of connections on $Y_1$, and $p(\cdot)$ is a path of 0-forms on $Y_1$. Fixing $\tau$, the failure of $\Pi$ to be smooth implies that the connection ${a}(\tau)$ will not be continuous on $Y_1$, unless 
					\begin{equation}\label{nuequation}
					\iota_\nu A = 0;
					\end{equation} 
					here $\nu$ is the normal vector to the hypersurface $\bb{R} \times \partial H \subset \bb{R} \times H^\infty$. However, by performing a suitable gauge transformation to $A$, we can always achieve (\ref{nuequation}). (See the previous section for a similar construction; also note that the action of the gauge group on $A$ does not change the value of (\ref{ener}).) When (\ref{nuequation}) holds it follows that the connection $a(\tau)$
				\begin{itemize}
				\item is continuous everywhere on $Y_1$, 
				\item is smooth away from the hypersurface $\left\{0, \pi \right\} \times \partial H \subset Y_1$, and 
				\item has bounded derivative near this hypersurface. 
				\end{itemize}
				In particular, ${a}(\tau)$ is of Sobolev class $H^1$ on $Y_1$.

					Now we introduce a convenient reference connection $a_0$ on $Y_1$ with which we will define $\CS_{a_0}$. This reference connection will depend on the given connection $A$; we continue to assume that (\ref{nuequation}) holds. Define $a_0$ on the first copy of $H$ in $Y_1$ by declaring it to equal $\smash{A \vert_{Y_0}}$, where we are identifying $Y_0$ with $H$ in the obvious way. Define $a_0$ on the second copy of $H$ to also equal $\smash{A \vert_{Y_0}}$. It remains to define $a_0$ on the cylinder $\left[0, \pi \right] \times \partial H$, and there is a unique way to do this if we require that $a_0$ is (i) continuous and (ii) constant in the $\left[0, \pi \right]$-direction. It follows from (\ref{nuequation}) that $a_0$ is of Sobolev class $H^1$. Moreover, 
					$$\lim_{\tau \rightarrow 0^+}  {a}(\tau) = a_0,$$ 
					where this limit is in the $H^1$-topology on $Y_1$ (this is basically just the statement that $A$ is continuous at $Y_0 \subset \bb{R} \times H^\infty$). Note that this choice of $a_0$ may not be flat. However, it turns out that $\CS_{a_0} = \CS_{a_1}$ for some flat connection $a_1$ (in fact, any flat connection in the diagonal (\ref{diagon}) will do); see Remark \ref{flatremarka}. 
					
\medskip					
					
					Now we prove Corollary \ref{cor}. At this stage the argument follows essentially as in \cite[Theorem 1.1]{Wenergyidentity}; we recall the details for convenience. Let $A$ be any finite energy connection on $\bb{R} \times H^\infty$, and assume it has been put in a gauge so that (\ref{nuequation}) holds. Use the identity $F_{\Pi^*A} = F_{{a}} + d \tau \wedge (\partial_\tau {\afa}  - d_{{\afa}} {p})$ to get 
					$$\frac{1}{2} \Pi^* \langle F_{A} \wedge F_{A} \rangle = d \tau \wedge \langle F_a \wedge  (\partial_\tau {\afa}  - d_{{\afa}} {p}) \rangle.$$ 
					Integrate both sides and use the fact that the image of $\Pi$ has full measure in $\bb{R} \times H^\infty$ to get
					\begin{equation}\label{44r}
					\begin{array}{rcl}
					 \fracd{1}{2} \intd{\bb{R} \times H^\infty} \langle F_A \wedge F_A \rangle &= &  \intdd{0}{\infty} \intd{Y_1} d \tau \wedge \langle F_{{a}} \wedge \partial_\tau {a} \rangle \\
					\medskip
					& =&  \intdd{0}{\infty} \frac{d}{d \tau} \CS_{{a}_0}({a}(\tau)) \: d\tau \\
					\medskip
					&=& \limd{\tau \rightarrow  \infty} \CS_{a_0}({a}(\tau)) - \limd{\tau \rightarrow 0^+} \CS_{a_0}({a}(\tau)),
					\end{array}
					\end{equation}
where we used the Bianchi identity to kill off the $d_{{\afa}} {p}$-term, and then used the definition of $\CS_{a_0}$. From the definition of $a_0$, we have
$$\limd{\tau \rightarrow 0^+} \CS_{a_0}({a}(\tau)) = \CS_{a_0}(a_0) = 0, $$
so it suffices to consider the limit at $\infty$. 

Notice that (\ref{44r}) shows that the $\lim_{\tau \rightarrow  \infty} \CS_{a_0}({a}(\tau))$ exists. The goal now is to show that this limit equals $\CS_{a_0}(a_\flat)$ for some flat connection $a_\flat$. Endow $Y_1 $ with the metric induced from $ds^2 + g$ via the inclusion $Y_1 \subset \bb{R} \times H^\infty$. Then it follows from the definitions that
$$\intdd{1}{\infty} \Vert F_{{\afa}(\tau)} \Vert^2_{L^2(Y_1)} \leq \Vert F_A \Vert_{L^2(\bb{R} \times H^\infty)}^2.$$
Since the energy of $A$ is finite, the integral over $\left[1, \infty \right)$ on the left converges and so there is a sequence $\tau_i \in \bb{R}$ with
$$\Vert F_{{\afa}(\tau_i)} \Vert^2_{L^2(Y_1)} \stackrel{i}{\longrightarrow} 0 \indent \textrm{and} \indent \tau_i \stackrel{i}{\longrightarrow} \infty.$$
By Uhlenbeck's weak compactness theorem \cite{U}, we can find
\begin{itemize}
\item a subsequence of the $\left\{{a}(\tau_i) \right\}$, denoted by $\left\{{a}_i \right\}$, 
\item a sequence of gauge transformations $\left\{u_i \right\}$, and 
\item a flat connection $a_\infty$, 
\end{itemize} 
for which $\left\{u_i^* {a}_i\right\}$ converges to $a_\infty$ weakly in $H^1$ and hence strongly in $L^4$. This convergence is enough to put each $u_i^*{a}_i$ in Coulomb gauge with respect to $a_\infty$ \cite[Theorem 8.1]{Wuc}, so by redefining each $u_i$ we may assume this is the case. Then $u_i^*{a}_i$ converges to $a_\infty$ strongly in $H^1$. Since $\CS_{a_0}$ is continuous in the $H^1$-topology, we have
$$\lim_{i \rightarrow \infty}  \CS_{a_0}(u_i^*{a}_i) = \CS_{a_0}(a_\infty).$$
On the other hand, 
$$\CS_{a_0}(u_i^*{a}_i) - \CS_{a_0}({a}_i) = \kappa(P_{u_i}) \in \bb{Z}$$
for all $i$. Since $\CS_{a_0}(u_i^*{a}_i)$ and $\CS_{a_0}({a}_i)$ both converge, it follows that $\kappa(P_{u_i})$ is constant for all but finitely many $i$. By passing to yet another subsequence, we may assume that $\kappa(P_{u_i})$ is constant for all $i$. Then there is some gauge transformation $u$ such that $\kappa(P_{u}) = \kappa(P_{u_i})$ for all $i$ (just take $u$ to be one of the $u_i$). This gives
$$\begin{array}{rcccl}
\fracd{1}{2} \intd{\bb{R} \times H^\infty} \langle F_A \wedge F_A \rangle & = & \limd{i\rightarrow \infty} \CS_{a_0}({a}_i) & = & \limd{i\rightarrow \infty}  \CS_{a_0}(u^*_i {a}_i) - \kappa(P_{u_i}) \\
\medskip
& =& \CS_{a_0}(a_\infty) - \kappa(P_u) & = & \CS_{a_0}((u^{-1})^*a_\infty).
\end{array}$$
So taking $a_\flat \defeq (u^{-1})^* a_\infty$ finishes the proof.

\begin{remark}		\label{flatremarka}
Here we address the fact that the reference connection $a_0$, constructed in the proof above, may not be a flat connection. We address this from two different angles. First of all, the quantity (\ref{ener}) is independent of the choice of connection $A$, provided that one restricts to connections with the same asymptotic behavior at infinity. In particular, one can always modify the connection $A$ so that its restriction to $Y_0$ is flat. This forces $a_0$ to be flat. 

Secondly, the argument of the previous paragraph suggests that the value $\CS_{a_0}(a)$ is somehow independent of $a_0$. It is interesting to see this explicitly without modifying the original connection $A$. There is an obvious $\bb{Z}_2$ action on $Y = \overline{H} \cup_{\partial H} H$ given by interchanging the two $H$-factors. Call a form or connection on $Y$ \emph{symmetric} if it is fixed by this action. For example, all elements of the diagonal (\ref{diagon}) are symmetric. The key observation here is that $a_0$ is symmetric. Then we claim that function $\CS_{a_0}$ is independent of the choice of $a_0$ from the class of symmetric connections. Indeed, suppose $a_1$ is a second connection that is symmetric. We want to show that $\CS_{a_0}(a) = \CS_{a_1}(a)$ for all connections $a$. From the definition of the Chern-Simons functional we have
$$\CS_{a_0}(a) - \CS_{a_1}(a) = -\CS_a(a_0) + \CS_a(a_1).$$
Note that the right-hand side is actually independent of $a$, since changing the connection $a$ changes $\CS_a$ by a constant. We can therefore replace $a$ with $a_0$ on the right-hand side to get 
$$\CS_{a_0}(a) - \CS_{a_1}(a)  = \CS_{a_0}(a_1) = \intd{Y}  \langle F_{a_0} \wedge v \rangle +  \frac{1}{2} \langle d_{a_0} v \wedge v \rangle  + \frac{1}{6} \langle \left[ v \wedge v \right] \wedge v \rangle,$$
where $v \defeq a_1 - a_0$. Let $cs_{a_0}(v)$ denote the integrand on the right. Now use the following facts: (i) $Y$ decomposes into two copies of $H$, (ii) the two copies of $H$ have opposite orientations, and (iii) $cs_{a_0}(v)$ is symmetric (it is made up of the symmetric $a_0, a_1$). These allow us to compute
$$\CS_{a_0}(a) - \CS_{a_1}(a) = \intd{\overline{H}} cs_{a_0}(v) + \intd{H} cs_{a_0}(v) = -\intd{{H}} cs_{a_0}(v) + \intd{H} cs_{a_0}(v)  = 0.$$
\end{remark}

\end{document}